\documentclass[12pt]{article}
\usepackage{amssymb,amsfonts, amsmath,amsthm,epsfig}

\newtheorem{theorem}{Theorem}[section]

\newtheorem{problem}[theorem]{Problem}

\newtheorem{conjecture}[theorem]{Conjecture}

\newcommand{\Sz}{\hbox{\rm Sz}\,}

\title{Mathematical aspects of Wiener index}

\author
{ Martin Knor\thanks{Slovak University of Technology in Bratislava,
Faculty of Civil Engineering, Department of Mathematics, Bratislava,
Slovakia. E-Mail: \texttt{knor@math.sk}},\quad Riste
\v{S}krekovski\thanks{FMF, University of Ljubljana \& Faculty of Information Studies, Novo mesto \&
Institute of Mathematics, Physics and Mechanics, Ljubljana \&
University of Primorska, FAMNIT, Koper, Slovenia. E-Mail:
\texttt{skrekovski@gmail.com}},\quad Aleksandra Tepeh\thanks{Faculty
of Information Studies, Novo mesto \& Faculty of Electrical
Engineering and Computer Science, University of Maribor, Slovenia.
E-Mail: \texttt{aleksandra.tepeh@gmail.com}} }

\begin{document}

\maketitle

{\abstract
{
The Wiener index (i.e., the total distance or the transmission number),
defined as the sum of distances between all unordered pairs of vertices
in a graph, is one of the most popular molecular descriptors.
In this article we summarize some results, conjectures and problems
on this molecular descriptor, with emphasis on works we were involved in.
}}


\section{Introduction}

Having a molecule, if we represent atoms by vertices and bonds by edges, we
obtain a molecular graph, \cite{Pogliani,Randic}.
Graph theoretic invariants of molecular graphs, which predict properties of
the corresponding molecule, are known as topological indices.
The oldest topological index is the Wiener index \cite{Wiener},
which was introduced in $1947$ as the path number.

At first, Wiener index was used for predicting the boiling points of
paraffins \cite{Wiener}, but later strong correlation between Wiener index and
the chemical properties of a compound was found.
Nowadays this index is a tool used for preliminary screening
of drug molecules \cite{Arg-00}.
Wiener index  also predicts binding energy of protein-ligand complex
at a preliminary stage.

Hence, Wiener index was used by chemists decades before it attracted
attention of mathematicians.
In fact, it was studied long time before the branch of discrete mathematics,
which is now known as Graph Theory, was developed.
Many years after its introduction, the same quantity has been studied
and referred to by mathematicians as the gross status \cite{Harary},
the distance of graphs \cite{EJS} and the transmission \cite{Soltes}.
A great deal of knowledge on Wiener index is accumulated in several survey papers \cite{surv1,petra,DM3,KS_chapter,Gut-sur}. This paper is also of similar kind and it appears in a volume dedicated
to A. Graovac, whose wide research opus of mathematical chemistry includes also works of Wiener index, e.g., see~\cite{Grao4,ori,Grao1,Grao2,Grao3}.

Let $d(u,v)$ denote the distance between vertices $u$ and $v$ in $G$. \emph{Wiener index} of a graph $G$, denoted by $W(G)$, is the sum
of distances between all (unordered) pairs of vertices of $G$
\begin{equation}
\label{WienerIndex}
     W(G)=\displaystyle\sum_{\{u,v\} \subseteq V(G)} d(u,v).
\end{equation}

Though Wiener index is the most common topological index,
nowadays we know over 200 topological indices used in chemistry.
Here we mention three of them, those, which can be considered
as weighted versions of Wiener index.

For an edge $e = ij$, let $n_e(i)$ be the number of vertices
of $G$ being closer to $i$ than to $j$ and let $n_e(j)$ be the
number of vertices of $G$ lying closer to $j$ than to $i$.
The {\em Szeged} index of a graph $G$ is defined by
$$
\Sz(G) =  \sum_{e=ij\in E(G)} n_e(i) n_e(j).
$$
This invariant was introduced by Gutman~\cite{G94} during his stay
at the Attila Jozsef University in Szeged, and he named it after
this place.

In 1989, lead by the idea of characterizing the alkanes,
Schultz~\cite{Schultz} defined a new index ${\rm MTI}(G)$ that is
degree and distance based. Gutman decomposed this index into two
parts and called one of them {\it Schultz index (of the first kind)},
which is defined by
$$
S(G)=\sum_{\{u,v\} \subseteq V(G)}\left(d(u)+d(v)\right)d(u,v)\,,
$$
where $d(v)$ denotes the degree of $v$. The same invariant was independently and simultaneously introduced by
Dobrynin and Kochetova~\cite{DK}.

Gutman~\cite{g-spsmti-94} also introduced a new index,
$$
\mathrm{Gut}(G) = \displaystyle\sum_{\{u,v\} \subseteq V(G)} d(u) d(v) d(u,v)\,,
$$
and named it the {\it Schultz index of the second kind}.
Nowadays this index is also known as the {\it Gutman} index.

In this paper we consider mathematical aspects of Wiener index.
This is not a typical survey.
We summarize our mathematical work on this molecular descriptor
over the past years and, what is more important, we integrate some
conjectures, problems, thoughts, and ideas for possible future work
that we find interesting.
We include also a couple of related open problems that have been considered
by other authors.


\section{Some fundamental properties of Wiener index}

Already in 1947, Wiener has shown that Wiener index of a tree
can be decomposed into easily calculable edge--contributions.
In what follows, by $n(G)$ we denote the number of vertices of $G$.
Let $F$ be a graph with $p$ components, $T_1,T_2,\dots,T_p$.
Then we set
$$
N_2(F) = \sum_{1 \leq i < j \leq p} n(T_i)\,n(T_j) \, .
$$
If $p=1$, that is if $F$ is connected, then $N_2(F)=0$.

\begin{theorem}[Wiener, 1947]
\label{W.th}
For a tree $T$ the following holds
  \begin{equation}  \label{W.eq}
    W(T) =  \sum_{e\in E(T)} N_2(T-e).
  \end{equation}
\end{theorem}


Since $T$ is a tree, for every edge $e=ij$ of $T$, the forest $T-e$ is comprised of two components,
one of size $n_e(i)$ and the other of size $n_e(j)$, which gives $N_2(T-e) = n_e(i) n_e(j)$.
Thus, one can restate (\ref{W.eq}) as
\begin{equation}
\label{W.eq2}
  W(T) =  \sum_{e=ij\in E(T)} n_e(i) n_e(j).
\end{equation}
So the Szeged and Wiener indices coincide on trees.
In fact, the Szeged index was defined from (\ref{W.eq2}) by relaxing
the condition that the graph is a tree.

In analogy to the classical Theorem~\ref{W.th}, we have the following vertex
version (see \cite{GS}):

\begin{theorem}
\label{Wv.th2}
Let $T$ be a tree on $n$ vertices.
Then
\begin{equation}
\label{vw}
  W(T) = \sum_{v \in V(T)} N_2(T-v) + \binom{n}{2} \,.
\end{equation}
\end{theorem}

An interesting theorem of Doyle and Graver~\cite{DG} is of a similar kind.
Let $F$ be a graph with $p$ components, $T_1,T_2,\dots,T_p$.
Then we set
$$
N_3(F) = \sum_{1 \leq i < j < k \leq p} n(T_i)\,n(T_j)\,n(T_k)\,.
$$
Note that if $p=1$ or $p=2$, then $N_3(F)=0$.  Now we state their
result, moreover we give an alternative short proof in the spirit of
combinatorial countings, more precisely involving combinations of
order 3.

\begin{theorem}[Doyle and Graver]
\label{thmDG}
Let $T$ be a tree on $n$ vertices.
Then
\begin{equation}
  W(T) + \sum_{v \in V(T)} N_3(T-v) = \binom{n+1}{3}\,.
\end{equation}
\end{theorem}
\begin{proof}
Let $V(T)=\{1,\ldots,n\}$ and let $V^*=V(T) \cup \{\pi\}$. For any path
$P=a_0a_1a_2\cdots a_k$ in $T$ with $a_0<a_k$, assign the 3-set $\{a_0,a_i,a_k\}$
to the edge $a_ia_{i+1}$ for $1\le i<k$ and $\{a_0,a_k,\pi\}$
to the edge $a_0a_1$. So we assign $k$ distinct 3-sets to a path $P$ of length $k$.
This way we assign all together $W(T)$ $3$-sets. For any non-assigned 3-set
$\{a,b,c\}$ (observe that $\pi$ does not appear here) of $\binom{V^*}{3}$, $T$ has no path containing them so there is precisely one vertex $v$ (their median) that this 3-set contributes 1 to  $N_3(T-v)$. As $V^*$ is of size $n+1$, the claim is established.
\end{proof}

It is well known that for trees on $n$ vertices, the maximum
Wiener index is obtained for the path $P_n$, and the minimum
for the star $S_n$.
Thus, for every tree $T$ on $n$ vertices we have
$$
(n-1)^2 = W(S_n) \le W(T) \le W(P_n) = {n+1 \choose 3}\,.
$$
Since the distance between any two distinct vertices is at least one,
among all graphs on $n$ vertices $K_n$ has the smallest Wiener index.
So for any connected graph $G$ on $n$ vertices, it holds
$$
{n \choose 2} = W(K_n) \le W(G) \le W(P_n)={n+1 \choose 3}\,.
$$
Note that the alternative proof of Theorem~\ref{thmDG} gives us a new proof that
$W(P_n)={n+1 \choose 3}$ and that $P_n$ is the extremal graph for the maximum.
Among 2-connected graphs on $n$ vertices (or even stronger, among the graphs
of minimum degree 2), the $n$-cycle has the largest Wiener index
$$
W(C_n) =
\begin{cases}
\frac{n^3}{8} & \mbox{if $n$ is even}, \\
\frac{n^3-n}{8} &  \mbox{if $n$ is odd}. \\
\end{cases}
$$


\section{The inverse Wiener index problem}

In 1995 Gutman and Yeh~\cite{GY} considered an inverse Wiener index problem.
They asked for which integers $w$ there exist trees with Wiener index $w$,
and posed the following conjecture:

\begin{equation}
\label{conj1}
  \mbox{\it For all but finitely many integers } w \mbox{ \it there exist trees with
  Wiener index } w.
\end{equation}

Inspired by the conjecture above, Lepovi\'c and Gutman~\cite{LG}
checked integers up to $1206$ and found $49$ integers that are not
Wiener indices of trees.
In 2004, Ban, Bereg, and Mustafa~\cite{BBM} computationally proved
that for all integers $w$ on the interval from $10^3$ to $10^8$
there exists a tree with Wiener index $w$.
Finally, in 2006, two proofs of the conjecture were published.
First, Wang and Yu~\cite{WY} proved that for every $w > 10^8$
there exists a caterpillar tree with Wiener index $w$.
The second result is due to Wagner~\cite{W}, who proved that
all integers but $49$ are Wiener indices of trees with diameter
at most $4$.

Surprisingly, it turns out that in most cases the inverse problem has many solutions.
Fink, Lu{\v z}ar and {\v S}krekovski~\cite{FLS} showed that the following theorem holds.
\begin{theorem}
\label{thm-many}
  There exists a function $f(w) \in \Omega(\sqrt[4]{w})$ such
  that for every sufficiently large integer $w$ there exist
  at least $2^{f(w)}$ trees with Wiener index $w$.
\end{theorem}
In~\cite{FLS} there is also proposed a constant time algorithm,
which for a given integer $w$ returns a tree with diameter four
and with Wiener index $w$. It would be interesting to find a better lower bound on $f(w)$ in
Theorem~{\ref{thm-many}}.

However, beside caterpillars and trees with small diameter, it could be interesting to
find some other types of trees (or graphs) that solve the inverse Wiener index problem.
Li and Wang~\cite{Li} considered this problem for peptoids,
Wagner et. al \cite{W3} for molecular and so-called hexagon type graphs, and Wagner \cite{W2}
for graphs with small cyclomatic number.

Bereg and Wang~\cite{BW} experimentally came to the observation
that this may hold for binary trees, as stated bellow.
Moreover, they observed that the conjecture may hold even when restricting
to $2$-trees, and even more, they where not able to disprove it for $1$-trees
(a binary tree of height $h$ is a {\em $k$-tree} if every vertex of depth
less than $h-k$ has precisely two children).

\begin{conjecture}
    Except for some finite set, every positive integer is the Wiener index of a binary tree.
\end{conjecture}

In~\cite{KrncS}  was considered the following problem, so called the {\em Wiener inverse interval problem}.

\begin{problem}
\label{prob:inv_gen}
For given $n$, find all values $w$ which are Wiener indices of graphs
(trees) on $n$ vertices.
\end{problem}
Regarding the above problem, let ${\rm WG}(n)$ and ${\rm WT}(n)$ be the
corresponding sets of values $w$ for graphs and trees on $n$ vertices, respectively.
Both sets have $\binom{n+1}{3}$ for the maximum element.
The smallest value in ${\rm WG}(n)$ is $\binom{n}{2}$ and in ${\rm WT}(n)$ it is $(n-1)^2$. In ~\cite{KrncS}, the size of the set ${\rm WG}(n)$ was considered, and it was shown that it is of order $\frac{1}{6}n^{3}+O\left(n^{2}\right)$. In the same paper were
stated the following problems.

\begin{conjecture}
The cardinality of ${\rm WG}(n)$ is of order $\frac{1}{6}n^{3}-\frac{1}{2}n^{2}+\Theta(n)$.
\end{conjecture}

\begin{conjecture}
The cardinality of ${\rm WT}(n)$ equals $\frac{1}{6}n^{3}+\Theta\left(n^{2}\right)$.
\end{conjecture}

In fact in~\cite{KrncS} it was shown that the length of the largest interval of integers which is fully
contained in ${\rm WG}(n)$ is of size $\frac{1}{6}n^{3}+O\left(n^{2}\right)$.
Regarding the length of the largest interval when only trees are considered, the following is conjectured.

\begin{conjecture}
In the set ${\rm WT}(n)$, the cardinality of the largest
interval of integers equals $\Theta\left(n^{3}\right)$.
\end{conjecture}

%
%
%

\section{Graphs with prescribed minimum/maximum degree}

Here we consider extremal values of Wiener index in some subclasses of the
class of all graphs on $n$ vertices. Recall that the {\em maximum degree} of a graph $G$, denoted by $\Delta(G)$, and the {\em minimum degree} of a graph, denoted by $\delta(G)$, are the maximum and minimum degree of its vertices.
As mentioned above, among $n$-vertex graphs with the minimum degree $\ge 1$,
the maximum Wiener index is attained by $P_n$.
But when restricting to minimum degree $\ge 2$, the extremal graph is $C_n$.
Observe that with the reasonable assumptions $\Delta\ge 2$ and $\delta\le n-1$,
the following holds
\begin{eqnarray}
&&\max\{W(G);\,
G\mbox{ has maximum degree at most }\Delta \mbox{ and }n\mbox{ vertices}\}
=W(P_n),\mbox{\qquad and}\nonumber\\
&&\min\{W(G);\, G\mbox{ has minimum degree at least }\delta \mbox{ and }n\mbox{ vertices}\}
=W(K_n).\nonumber
\end{eqnarray}

Analogous reasons motivate the following two problems.

\begin{problem}
\label{prob:up-bound}
What is the maximum Wiener index among $n$-vertex graphs with the minimum degree
at least $\delta$?
\end{problem}

\begin{problem}
\label{prob:lo-bound}
What is the minimum Wiener index among $n$-vertex graphs with the maximum degree
at most $\Delta$?
\end{problem}

A related problem was considered by Fischermann et al.~\cite{fish},
and independently by Jelen and Trisch in~\cite{jel,tri}, who
characterized the trees which minimize the Wiener index among all
$n$-vertex trees with the maximum degree at most $\Delta$. They also
determined the trees which maximize the Wiener index, but in a much
more restricted family of trees which have two distinct vertex
degrees only. Later Stevanovi\'{c}~\cite{stev} determined the trees
which maximize the Wiener index among all graphs with the maximum
degree $\Delta$, and originally Problem  \ref{prob:lo-bound} was
proposed by him in an equivalent form which requires that the
maximum degree is precisely $\Delta$.

Restricting to $\Delta=\delta=r$, i.e., restricting to regular graphs, could be
especially interesting.
In general, introducing (resp. removing) edges in a graph decreases
(resp. increases) the Wiener index, but in the class of $r$-regular graphs
on $n$ vertices we have fixed number of $r\cdot n/2$ edges.
Thus, more important role is played by the diameter.
Recall that in the case of trees, where the number of edges is
fixed as well, maximum Wiener index is attained by $P_n$ which has the largest
diameter, and minimum Wiener index is attained by $S_n$, which has
the smallest diameter.
Let us start with the first nontrivial case $r=3$, i.e. with cubic graphs.

\begin{figure}[h]
\begin{center}
\epsfig{file=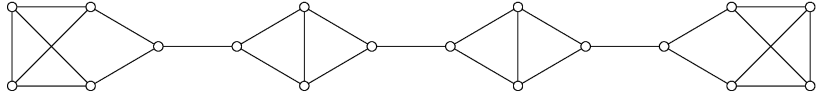, height=14mm}
\caption{The graph $L_{18}$.}
\label{fig:L18}
\end{center}
\end{figure}

Let $n$ be even and $n\ge 10$.
If $4\nmid n$, then $L_n$ is obtained from $(n-10)/4$ copies of $K_4-e$
joined to a path by edges connecting the vertices of degree $2$,
to which at the ends we attach two pendant blocks, each on $5$ vertices,
see Figure~{\ref{fig:L18}} for $L_{18}$.

\begin{figure}[h]
\begin{center}
\epsfig{file=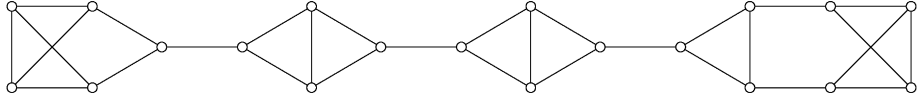, height=14mm}
\caption{The graph $L_{20}$.}
\label{fig:L20}
\end{center}
\end{figure}

On the other hand if $4\mid n$, then $L_n$ is obtained from $(n-12)/4$
copies of $K_4-e$, joined into a path by edges connecting the
vertices of degree $2$, to which ends we attach two pendant blocks, one on $5$
vertices and the other on $7$ vertices, see Figure~{\ref{fig:L20}} for
$L_{20}$ \cite{BalabanWe}. We have the following conjecture.

\begin{conjecture}
\label{conj:3-bound}
Among $n$-vertex cubic graphs, $L_n$ has the largest Wiener index.
\end{conjecture}

We believe that similar statements hold also for higher $r\ge 4$, with intermediate
repetitive gadget $K_4-e$ replaced by $K_{r+1}-e$, where on both ends we attach
suitable gadgets so that the resulting graph will have $n$ vertices.
Actually, these graphs are those with the maximum diameter, see~\cite{Kdmgt},
where the problem of finding a regular graph of given order and degree with
maximum diameter is studied from a different point of view.

The cubic graphs with the minimum Wiener index are hard to describe for us but
it seems that they have the smallest diameter.
For suitable $n$, good candidates are the cage graphs, e.g. Petersen graph and Heawood graph.
Guided by our intuition, we believe that the following may hold.

\begin{conjecture}
\label{c.max_diameter}
Among all $r$-regular graphs on $n$ vertices, the maximum Wiener index
is attained by a graph with the maximum possible diameter.
\end{conjecture}

\begin{conjecture}
\label{c.max_diameter}
Among all $r$-regular graphs on $n$ vertices, the minimum Wiener index
is attained by a graph with the minimum possible diameter.
\end{conjecture}

%
%

\section{Graphs with prescribed diameter/radius}

The {\em eccentricity} ${\rm ecc}(v)$ of a vertex $v$ in $G$ is the largest distance from $v$ to
another vertex of $G$; that is, ${\rm max}\{d(v,w) \,\vert\, w\in V(G)\}$.
The {\em diameter} of $G$, denoted by ${\rm diam}(G)$,
is the maximum eccentricity in $G$.
Similarly, the {\em radius} of $G$, denoted by ${\rm rad}(G)$,
is the minimum eccentricity in $G$.

Plesn{\'\i}k \cite{P} obtained graphs with minimum Wiener index in the class of graphs of order $n$ and diameter $d$ ($d\leq n-1$). When $d<n-1$, they are cycle-containing graphs. In 1975 he \cite{P75} addressed the following problem.

\begin{problem}
What is the maximum Wiener index among graphs of order $n$ and
diameter $d$?
\end{problem}

This problem remains unsolved even under additional restrictions.
DeLaVi{\~n}a and Waller \cite{DeLa} conjectured the following.

\begin{conjecture}
Let $G$ be a graph with diameter $d>2$ and order $2d + 1$. Then
$W(G) \le W(C_{2d+1})$, where $C_{2d+1}$ denotes the cycle of length
$2d+1$.
\end{conjecture}

Wang and Guo \cite{wang2} determined the trees with maximum Wiener index among trees of order $n$
and diameter $d$ for some special values of $d$, $2\leq d\leq 4$ or $n-3\leq d \leq n-1$. Independently, Mukwembi \cite{muk14} considered the diameter up to $6$ and showed that bounds he obtained are best possible. To find a sharp upper bound on the Wiener index for trees of given order and larger diameter could also be interesting.

For any connected graph $G$, ${\rm rad}(G)\leq {\rm diam}(G) \leq 2\, {\rm rad}(G)$.
By considering the close relationship between the diameter
and the radius of a graph, it is natural to consider the
above problem with radius instead of diameter~\cite{Chen}.

\begin{problem} \label{n-r} What is the maximum Wiener index among
graphs of order $n$𝑛 and radius $r$?
\end{problem}

Chen et al. \cite{Chen} characterized graphs with the maximum Wiener
index among all graphs of order 𝑛$n$ with radius two.
Analogous problem for the minimum Wiener index was posed by You and Liu \cite{You}.

\begin{problem} What is the minimum Wiener index among all graphs of order $n$𝑛 and radius $r$?
\end{problem}

Regarding this problem, Chen et al. \cite{Chen} stated the following conjecture.
For integers $n$, $r$, and $s$ with $n\geq 2r$, $r\ge 3$, and $n-2r+1\ge s\ge 1$, construct
a graph $G_{n,r,s}$ from a $2r$-cycle $v_1v_2\cdots v_{2r}$ so that
$v_1$ is replaced by $K_s$ and $v_2$ is replaced by $K_{n-2r+2-s}$,  connect $v_{2r}$  to each vertex of $K_s$, connect each vertex of $K_s$  to each vertex of $K_{n-2r+2-s}$, and connect each vertex of
$K_{n-2r+2-s}$ to $v_3$ (in other words $v_1$ is replicated  $s-1$ times, and
$v_2$ is replicated $n-2r+1-s$ times).  Notice that the resulting graph has $n$ vertices and radius $r$.

\begin{conjecture}
Let $n$ and $r$ be two positive integers with $n\geq 2r$ and $r\ge 3$. Then graphs $G_{n,r,s}$
for $s\in \{1,\ldots,r-1\}$ attain the minimum Wiener index in the class of graphs on $n$ vertices and with radius $r$.
\end{conjecture}


\section{Congruence relations for Wiener index}

It was of interest to several authors to obtain congruence relations
for the Wiener index. The first result of this kind was proved by
Gutman and Rouvray \cite{GR-90}. They established the congruence
relation for the Wiener index of trees with perfect matchings.

\begin{theorem}[Gutman and Rouvray]
\label{thm:gut-rouv}
Let $T$ and $T'$ be two trees on the same number of vertices.
If both $T$ and $T'$ have perfect matchings, then
$W(T) \equiv W(T')\pmod 4$.
\end{theorem}

A {\em segment} of a tree is a path contained in the tree whose terminal vertices
are branching or pendant vertices of the tree.
Dobrynin, Entringer and Gutman \cite{surv1} obtained a congruence relation
for the Wiener index in the class of {\em $k$-proportional trees}.
Trees of this class have the same order, the same number of segments,
and the lengths of all segments are multiples of $k$.

\begin{theorem}[Dobrynin, Entringer and Gutman]
\label{thm:co_DEG}
Let $T$ and $T'$ be two $k$-proportional trees. Then
$$
W(T)\equiv W(T') \pmod{k^{3}}\,.
$$
\end{theorem}

Theorem \ref{thm:gut-rouv} was recently generalized by Lin in \cite{Lin}
by establishing the congruence relation for the Wiener index of
trees containing $T$-factors.
A graph $G$ has a {\em $T$-factor} if there exist vertex disjoint
trees $T_1,T_2,\ldots, T_p$ such that
$V(G)=V(T_1)\cup V(T_2)\cup \cdots \cup V(T_p)$ and each $T_i$
is isomorphic to a tree $T$ on $r$ vertices.
If $T$ is a path on $r$ vertices, we say that the graph
$G$ has a $P_r$-factor. In this sense the well-known perfect
matching is a $P_2$-factor.

\begin{theorem}[Lin]
\label{thm:lin} If $T$ and $T'$ are two trees on the same number of
vertices, both with $P_r$-factors, then
$$
W(T) \equiv W(T')\pmod r \textit{\hspace{0.3 cm} for odd $r$,}
$$
and
$$
W(T) \equiv W(T')\pmod{2r} \textit{\hspace{0.3 cm} for even $r$.}
$$
\end{theorem}

Recently Gutman, Xu and Liu \cite{Gut-14} showed that the first
congruence in the above result is a special case of a much more
general result on the Szeged index.
As its consequence, for the Wiener index they obtained
the following result.

\begin{theorem}[Gutman, Xu and Liu]\label{thm:gxy}
Let $\Gamma_0$ be the union of connected graphs $G_1,G_2,\ldots,
G_p$, $p\geq 2$, each of order $r\geq 2$, all blocks of which are
complete graphs. Denote by $\Gamma $ a graph obtained by adding
$p-1$ edges to $\Gamma_0$ so that the resulting graph is connected.
Then
$$
W(\Gamma) \equiv \displaystyle\sum_{i=1}^{p} W(G_i)\pmod r\,.
$$
\end{theorem}

In \cite{congWe} we generalized both the above results.
Let $r$ and $t$ be integers, $r\ge 2$ and $0\le t<r$.
Further, let $\mathcal H=\{H_1,H_2,\dots,H_{\ell}\}$ be a set of connected
graphs, such that for all $i$, $1\le i\le\ell$, we have
$|V(H_i)|\equiv -t\pmod r$.
Finally, let $\mathcal F=\{F_1,F_2,\dots,F_{\ell-1}\}$ be a set of
connected graphs, such that for all $j$, $1\le j\le\ell-1$, we have
$|V(F_j)|\equiv t+2\pmod r$.
For every $F_j$, choose vertices $v_j^1,v_j^2\in V(F_j)$.
We remark that the vertices $v^1_j$ and $v^2_j$ are not necessarily
distinct.
Denote by $\mathcal{G}=\mathcal{G}(\mathcal{H}, \mathcal{F})$ the set
of all graphs obtained when all the vertices $v_j^1$ and $v_j^2$,
$1\le j\le\ell-1$, are identified with some vertices of
$H_1\cup H_2\cup\dots\cup H_{\ell}$ so that the resulting graph is
connected.

Every graph in $\mathcal{G}$ contains $\ell$ graphs from $\mathcal H$,
$\ell-1$ graphs from $\mathcal F$, and each graph of $\mathcal F$
connects two graphs of $\mathcal H$.
Since the graphs in $\mathcal{G}$ are connected, if we contract every $H_i$
to a single vertex and we consider $F_j$'s as edges joining pairs of these
contracted vertices, then the resulting graph is a tree.
In this way, $H_1,H_2,\dots,H_{\ell}$ can be regarded as supervertices,
$F_1,F_2,\dots,F_{\ell-1}$ as superedges, and the resulting graph has
a tree structure.

In Figure~{\ref{fig:obluda}} we have one graph $G$ of $\mathcal{G}$ for
given parameters $r$, $t$ and $\ell$, and for given sets $\mathcal H$,
$\mathcal F$ and $\{v^1_j,v^2_j\}_{j=1}^{\ell-1}$.
The vertices of $H_j$'s are depicted by full circles in
Figure~{\ref{fig:obluda}} and the edges of $H_i$'s are thick.

\begin{figure}[hhh]
\begin{center}
\epsfig{file=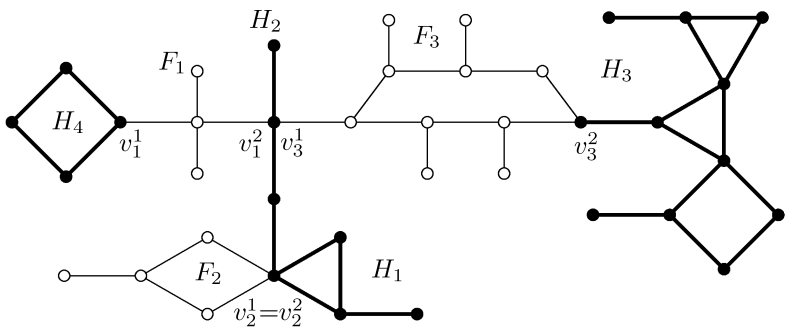,height=44mm}
\caption{A graph of $\mathcal{G}$ for $r=7$, $t=3$, $\ell=4$ and given
$H_i$'s, $F_j$'s and $v^k_j$'s.}
\label{fig:obluda}
\end{center}
\end{figure}

\begin{theorem}
\label{thm:co_main}
Let $G_1,G_2\in\mathcal{G}$.
Then $W(G_1)\equiv W(G_2)\pmod r$\,.
\end{theorem}

Now we generalize the second part of Theorem~{\ref{thm:lin}}.
Let $r$ be an even number, $r\ge 2$.
Further, let $\mathcal H=\{H_1,H_2,\dots,H_{\ell}\}$ be a set of trees,
such that for all $i$, $1\le i\le\ell$, we have $|V(H_i)|\equiv 0\pmod r$.
Finally, let $\mathcal F=\{F_1,F_2,\dots,F_{\ell-1}\}$ be a set of
trees, such that for all $j$, $1\le j\le\ell-1$, we have
$|V(F_j)|\equiv 2\pmod r$.
For every $F_j$, choose vertices $v_j^1,v_j^2\in V(F_j)$.
Denote by $\mathcal{G}^T$ the set of all graphs obtained when all the vertices
$v_j^1$ and $v_j^2$, $1\le j\le\ell-1$, are identified with some vertices
of $H_1\cup H_2\cup\dots\cup H_{\ell}$ so that the resulting graph is
connected.
Hence, $\mathcal{G}^T$ is a restriction of $\mathcal{G}$ when all the graphs
in $\mathcal H$ and $\mathcal F$ are trees and $t=0$.

\begin{theorem}
\label{thm:co_main2r}
Let $r$ be even and $G_1,G_2\in\mathcal{G}^T$.
Then $W(G_1)\equiv W(G_2)\pmod{2r}$\,.
\end{theorem}

Theorems~{\ref{thm:co_main}} and ~{\ref{thm:co_main2r}} show limits of
Theorem~{\ref{thm:lin}}.  Note that a segment can be defined on graphs
as well, and similarly one can define $k$-proportional graphs. So,
it would be interesting to find analogous limits for
Theorem~{\ref{thm:co_DEG}}.

\begin{problem}
Let $G$  and $G$ be two $k$-proportional graphs.
Under which conditions we have $W(G)\equiv W(G') \pmod{k^{3}}\,?$
\end{problem}


\section{Wiener index and the line graph operation}

Let $G$ be a graph. Its {\em line graph}, $L(G)$, has vertex set identical with the set of
edges of $G$ and two vertices of $L(G)$ are adjacent if and only if the corresponding edges are adjacent in $G$.
{\em Iterated} line graphs are defined inductively as follows:

\begin{displaymath}
L^i(G)=\left\{ \begin{array}{ll}
G & \textrm{if $i=0$,}\\
L(L^{i-1}(G)) & \textrm{if $i>0$.}
\end{array}\right.
\end{displaymath}

The main problem here is to determine the relation between $W(L(G))$ and
$W(G)$.
Particularly, we focuss on graphs $G$ satisfying
\begin{equation} \label{eq.WLW}
    W(L(G))=W(G),
\end{equation}
see \cite{dgms-ewig-09,DobryninM05,dm-wilgcn-05,GutmanE96,gp-mdlg-97}, and
in particular see the expository papers~\cite{DM3, KS_chapter}.
Let us remark that in the literature one easily encounters the term
{\em edge-Wiener index} of $G$, which is actually the Wiener index
of the line graph, sometimes shifted by ${n \choose 2}$, see \cite{KYAW}.

The following remark of Buckley~\cite{Buckley81} is a  pioneering work in this area.

\begin{theorem}[Buckley, 1981]
\label{Buckley.thm}
For every tree $T$, $W(L(T))=W(T)- {n \choose 2}$.
\end{theorem}


By the above result, Wiener index of a line graph of a tree is strictly smaller
than the Wiener index of the original tree.
An interesting generalization of this was given by Gutman~\cite{g-dlg-96}:

\begin{theorem}
If $G$ is a connected graph with $n$ vertices and $m$ edges, then
$$W(L(G)) \geq W(G)-n(n-1) + \frac{1}{2} m(m+1).$$
\end{theorem}

In addition, regarding Theorem~\ref{Buckley.thm}, Gutman and
Pavlovi{\' c}~\cite{gp-mdlg-97} showed that the Wiener index of a
line graph is not greater than the Wiener index of the original
graph even if we allow a single cycle in the graph.

\begin{theorem}
If $G$ is a connected unicyclic graph with $n$ vertices, then $W(L(G)) \leq W(G)$,
with equality if and only if $G$ is a cycle of length $n$.
\end{theorem}
In connected bicyclic graphs all the three cases $W(L(G)) < W(G)$, $W(L(G)) = W(G)$,
and $W(L(G)) > W(G)$ occur \cite{gp-mdlg-97}. There are
26 bicyclic graphs of order 9 with the property $W(L(G)) = W(G)$ \cite{dob1-97,gjd-sgdgedlg-97}, and already
166 ten-vertex vertices with this property, see \cite{DM3}.

The following result tells us that in most cases (\ref{eq.WLW})
does not hold for graphs of minimum degree at least 2, see \cite{CDKSV,Wu}.

\begin{theorem}\label{thm:1}
Let $G$ be a connected graph with $\delta(G)\geq 2$. Then
$$
W(L(G))\geq W(G).
$$
Moreover, the equality holds only for cycles.
\end{theorem}


\subsection{Sandwiching by Gutman index}

The following result was proved independently and simultaneously
in~\cite{CDKSV} and~\cite{Wu}.

\begin{theorem}
Let $G$ be a connected graph of size $m$.
Then
$$
  \frac{1}{4} ({\rm Gut}(G) - m) \le W(L(G)) \le
  \frac{1}{4} ({\rm Gut}(G) - m) + \binom{m}{2}\,.
$$
Moreover, the lower bound is attained if and only if G is a tree.
\end{theorem}

Let $\kappa_i(G)$ denote the number of $i$-cliques in a graph $G$.
In \cite{KPS-0}, the lower bound of the above theorem is improved in the
following way.

\begin{theorem}
\label{thm1}
Let $G$ be a connected graph. Then,
\begin{equation}
\label{eqn2}
 W(L(G))  \ge  \frac{1}{4}{\rm Gut}(G) - \frac{1}{4}|E(G)| + \frac{3}{4}\kappa_3(G) + 3\kappa_4(G)
\end{equation}
with the equality in (\ref{eqn2}) if and only if $G$ is a tree or a complete graph.
\end{theorem}

It follows from the above theorem that for a connected graph $G$ of minimal degree $\delta\ge 2$ we have
$$ W(L(G))  \ge  \frac{\delta^2}{4}W(G) - \frac{1}{4}|E(G)| \ge {\delta^2-\frac{1}{4}}W(G).$$

Moreover, this lower bound was improved in \cite{miAMC}.

\begin{theorem}
Let $G$ be a connected graph of minimum degree $\delta$. Then
  $$ W(L(G))\ge \frac{\delta^2}{4} W(G) $$
with equality holding if and only if $G$ is isomorphic to a path on three
vertices or a cycle.
\end{theorem}

\subsection{Extremal line graphs}

The problem of finding graphs, whose line graph has maximal Wiener index was given by
Gutman~\cite{g-dlg-96} (see also~\cite{DM3}).
\begin{problem}
Find an $n$-vertex graph $G$ whose line graph $L(G)$ has maximal Wiener index.
\end{problem}

We say that a graph is {\em dumbell} if it is comprised
of two disjoint cliques connected by a path, and similarly a graph is {\em barbell} if it is comprised
of two disjoint complete bipartite graphs connected by a path.

\begin{conjecture}
In the class of graphs $G$ on $n$ vertices, $W(L(G))$ attains maximum for some dumbell graph.
\end{conjecture}

The above conjecture is supported by the result in ~\cite{dgms-ewig-09}. We state a similar one for bipartite graphs.

\begin{conjecture}
Let $n$ be a large integer. Then in the class of all bipartite graphs $G$ on $n$ vertices $W(L(G))$ attains maximum for some barbell graph.
\end{conjecture}

%
%

\subsection{Extremal ratios}

Dobrynin and Mel'nikov~\cite{DM3} proposed to estimate the extremal values for the ratio 
 \begin{equation} \label{ratio}
                   \frac{W(L^k(G))}{W(G)},
 \end{equation}             
and explicitly stated the case $k=1$ as a problem.  In \cite{miAMC} this problem was solved for the minimum.

\begin{theorem}
Among all connected graphs on $n$ vertices, the fraction $\frac{W(L(G))}{W(G)}$ is
minimum for the star $S_n$.
\end{theorem}

The problem for the maximum remains open.
\begin{problem}
 Find  $n$-vertex graphs $G$ with maximal values of $\frac{W(L(G))}{W(G)}$.
\end{problem}

Notice that
$$ \frac{W(L(S_n))}{W(S_n)} =  \frac{n - 2}{2(n + 1)}, \quad \frac{W(L(P_n))}{W(P_n)} =  \frac{n - 2}{n + 1},\quad
\hbox{ and } \quad\frac{W(L(K_n))}{W(K_n)} =  \binom{n-1}{2}.$$
The line graph  of $K_n$ has the greatest number of vertices, and
henceforth, it may attain the maximum value. Restricting to
bipartite graphs, the almost balanced complete bipartite graphs have
most vertices, so in this class of graphs the extreme could be
$K_{\lfloor n/2\rfloor, \lceil n/2\rceil}$. 

Regarding the minimum of (\ref{ratio}),  we expect that for higher iterations $k\ge 2$, it should be $P_n$, as it is the only
graph whose line graph decreases in size. We believe the following holds, it is proposed and considered in \cite{HKS}. 

\begin{conjecture}
Let $k\ge 2$ and let $n$ be a large integer. Then in the class of graphs $G$ on $n$ vertices $W(L^k(G))/W(G)$ attains the maximum for $K_n$, and
it attains the minimum for $P_n$.
\end{conjecture}


\subsection{Graphs with given girth}\label{sec:wg=wlg}

The {\em girth} of a graph is the length of a shortest cycle contained in the graph.
A connected graph $G$ is isomorphic to $L(G)$ if and only if $G$ is a cycle.
Thus, cycles provide a trivial infinite family of graphs for which
$W(G)=W(L(G))$.
In~\cite{DM}, Dobrynin and Mel'nikov stated the following problem.
\begin{eqnarray}
\label{problem20}
&&\mbox{\it Is it true that for every integer } g \geq 5 \mbox{ \it  there exists a graph }
G \neq C_g \mbox{ \it of girth } g,\mbox{\qquad}\nonumber\\
&&\mbox{\it for which } W(G) = W(L(G))?
\end{eqnarray}

The above problem~({\ref{problem20}}) was solved by Dobrynin~\cite{D5}
for all girths $g\ne\{5,7\}$; these last two cases were solved
separately.
Already in \cite{DM}, Dobrynin and Mel'nikov \cite{DM} constructed
infinite families of graphs of girths three and four
with the property $W(G)=W(L(G))$.
Inspired by their result the following statement was proved in
\cite{CDKSV}.

\begin{theorem}
\label{girth_main2}
For every non-negative integer $h$, there exist infinitely many
graphs $G$ of girth $g = h^2 + h + 9$ with $W(L(G)) = W(G)$.
\end{theorem}

The above result encouraged the authors of~\cite{CDKSV} to state
the following conjecture.

\begin{conjecture}
\label{problem30}
For every integer $g \geq 3$, there exist infinitely many graphs $G$
of girth $g$ satisfying $W(G) = W(L(G))$.
\end{conjecture}


\subsection{Graphs and cyclomatic number}

The {\em cyclomatic number} $\lambda(G)$ of a graph $G$ is defined
as $\lambda(G)=|E(G)|-|V(G)|+1$. Some attention was devoted to
graphs $G$ with prescribed cyclomatic number satysfying the equality
$W(L(G)) = W(G)$. As already mentioned, the smallest $26$ bicyclic
graphs with $9$ vertices are reported in
\cite{dob1-97,gjd-sgdgedlg-97}. Bicyclic graphs up to $13$ vertices
are counted in \cite{DM3} and diagrams of such graphs with $9$ and
$10$ vertices are also given. The smallest $71$ tricyclic graphs
with $12$ vertices are counted in \cite{dob1-97}. There are $733$
tricyclic graphs of order $13$ with this properties \cite{DM3}.
Denote by $n(\lambda)$ the minimal order of graphs with cyclomatic
number $\lambda \geq 2$ and $W(L(G)) = W(G)$. Then $n(2) = 9$ and
$n(3) = 12$.

Graphs with increasing cyclomatic number were constructed in
\cite{dob2-04,dm-wilgcn-05,DobryninM05}. To construct graphs from
\cite{DobryninM05}, properties of the Pell equation from the number
theory were applied. The cyclomatic number $\lambda$ of graphs from
\cite{DobryninM05} rapidly grows and the order of graphs is
asymptotically equal to $(2+\sqrt{ 5})\lambda \approx 4.236\lambda$
when $\lambda \rightarrow \infty$. The following conjecture was put
forward in \cite{DobryninM05}:

\begin{conjecture}
The graphs constructed in  \cite{DobryninM05} have the minimal order among all
graphs with given cyclomatic number satisfying the property $W(L(G)) = W(G)$.
\end{conjecture}

Graphs for all possible  $\lambda \geq 2$ were constructed in \cite{dm-wilgcn-05}. It is known that
$n(\lambda) \leq 5\lambda$ for $\lambda \geq 4$, $n(5) \leq 21$ and $n(7) \leq 29$. The following problem was
posed in \cite{dob1-97}.

\begin{problem}
Find an exact value of $n(\lambda)$ for small $\lambda \geq 4$.
\end{problem}


\subsection{Quadratic line graphs}

The graph $L^2(G)$ is also called the {\em quadratic line graph} of $G$.
As mentioned above, for non-trivial tree $T$ we cannot have $W(L(T))=W(T)$.
But there are trees $T$ satisfying
\begin{equation}
\label{eq:basic}
  W(L^2(T))=W(T),
\end{equation}
see \cite{D,DM1,DM2,KS_chapter}.
Obviously, the simplest trees are such which have a unique vertex of degree
greater than $2$.
Such trees are called {\em generalized stars}.
More precisely, {\em generalized $t$-star} is a tree obtained from the star
$K_{1,t}$, $t\ge 3$, by replacing all its edges by paths of positive
lengths, called {\em branches}.
In \cite{DM} we have the following theorem.

\begin{theorem}
\label{thm:DM}
Let $S$ be a generalized $t$-star with $q$ edges and branches of length
$k_1,k_2,\ldots,k_t$.
Then
\begin{equation}
\label{eq:general}
  W(L^2(S))=
  W(S)+\frac{1}{2}\binom{t-1}{2}\bigg(\sum_{i=1}^t k_i^2+q\bigg)-
  q^2 + 6\binom t4.
\end{equation}
\end{theorem}

Based on this theorem, it is proved in \cite{DM} that $W(L^2(S))<W(S)$ if
$S$ is a generalized $3$-star, and $W(L^2(S))>W(S)$ if $S$ is a generalized
$t$-star where $t\ge 7$.
Thus, property (\ref{eq:basic}) can hold for generalized $t$-stars only when
$t\in\{4,5,6\}$.
In \cite{DM} and \cite{KS_aust}, for every $t\in\{4,5,6\}$
infinite families of generalized $t$-stars with property (\ref{eq:basic}) were found, see also \cite{KS_chapter}.
These results suggest the following conjecture \cite{dob05}:

\begin{conjecture}
\label{conj:infL2}
Let $T$ be a non-trivial tree such that $W(L^2(T))=W(T)$.
Then there is an infinite family of trees $T'$ homeomorphic to $T$, such that
$W(L^2(T'))=W(T')$.
\end{conjecture}

Of course, more interesting is the question which types of trees satisfy (\ref{eq:basic}).
Perhaps such trees do not have many vertices of degree at least 3.
Let $\mathcal T$ be a class of trees which have no vertex of degree two, and such that
$T\in\mathcal T$ if and only if there exists a tree $T'$ homeomorphic to $T$, and such
that $W(L^2(T'))=W(T')$.

\begin{problem}
\label{prob:L2T}
Characterize the trees in $\mathcal T$.
\end{problem}

By the above results, among the stars only $K_{1,4}$, $K_{1,5}$, and
$K_{1,6}$ are in $\mathcal T$.
We expect that no tree in $\mathcal T$ has a vertex of degree exceeding $6$.
Based on our experience, we also expect that there is a constant $c$ such that
no tree in $\mathcal T$ has more than $c$ vertices of degree at least $3$.
Consequently, we believe that the following conjecture is true.

\begin{conjecture}
$\mathcal T$ is finite.
\end{conjecture}


\subsection{Iterated line graphs}

As we have seen, there is no non-trivial tree $T$ for which $W(L(T))=W(T)$
and there are many trees $T$, satisfying $W(L^2(T))=W(T)$.
However, it is not easy to find a tree $T$ and $i\ge 3$ such that
$W(L^i(T))=W(T)$.
In \cite{surv1}, Dobrynin, Entringer and Gutman posed the following problem:
\begin{equation}
\label{prob:ite_target}
\mbox{\it Is there any tree } T \mbox{ \it satisfying equality } W(L^i(T))=W(T)
\mbox{ \it for some } i\ge 3?
\end{equation}

Observe that if $T$ is a trivial tree, then $W(L^i(T))=W(T)$
for every $i\ge 1$, although here the graph $L^i(T)$ is empty.
The real question is, if there is a non-trivial tree $T$ and $i\ge 3$ such
that $W(L^i(T))=W(T)$.
The same question appeared four years later in~\cite{DM} as a conjecture.
Based on the computational experiments, Dobrynin and Mel'nikov expressed
their belief that the problem has no non-trivial solution and stated the
following conjecture:
\begin{equation}
\label{conj:ite_target}
\mbox{\it There is no tree }T\mbox{ \it satisfying equality }W(T)=W(L^i(T))
\mbox{ \it for any } i\ge 3.
\end{equation}

\begin{figure}[hhh]
\begin{center}
\epsfig{file=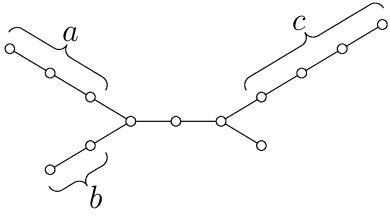, height=29mm}
\caption{The graph $H_{a,b,c}$.}
\label{fig:Habc}
\end{center}
\end{figure}

In a series of papers \cite{KPS4.5}, \cite{KPS5}, \cite{KPS1}, \cite{KPS2},
\cite{KPS3} and \cite{KPS4}, conjecture~(\ref{conj:ite_target})
was disproved and all solutions of problem~(\ref{prob:ite_target}) were found, see also~\cite{KS_chapter}.
Let $H_{a,b,c}$ be a tree on $a+b+c+4$ vertices, out of which two have
degree $3$, four have degree $1$ and the remaining $a+b+c-2$ vertices
have degree $2$.
The two vertices of degree $3$ are connected by a path of length $2$.
Finally, there are two pendant paths of lengths $a$ and $b$ attached to one
vertex of degree $3$ and two pendant paths of lengths $c$ and $1$ attached
to the other vertex of degree $3$, see Figure~{\ref{fig:Habc}} for $H_{3,2,4}$.
We have the following statement.

\begin{theorem}
\label{thm:H}
For every $j,k\in\mathbb Z$ define
\begin{eqnarray}
\label{eq:abc}
a&=&128+3j^2+3k^2-3jk+j,\nonumber\\
b&=&128+3j^2+3k^2-3jk+k,\nonumber\\
c&=&128+3j^2+3k^2-3jk+j+k.\nonumber
\end{eqnarray}
Then $W(L^3(H_{a,b,c}))=W(H_{a,b,c})$.
\end{theorem}

Let $\ell\in\{j,k,j+k\}$.
Since for every integers $j$ and $k$ the inequality $3j^2+3k^2-3jk+\ell\ge 0$ holds, we see that
$a,b,c\ge 128$ in Theorem~{\ref{thm:H}}.
Therefore, the smallest graph satisfying the assumptions is $H_{128,128,128}$
on $388$ vertices, obtained when $j=k=0$.
If we take in mind that there are approximately $7.5\cdot 10^{175}$ non-isomorphic trees
on $388$ vertices while the number of atoms in the entire Universe is estimated to be only
within the range of $10^{78}$ to $10^{82}$, then to find ``a needle in a haystack"
is trivially easy job compared to finding a counterexample when using only the brute
force of (arbitrarily many) real computers.

The following theorem gives a complete answer to
problem~({\ref{prob:ite_target}}).

\begin{theorem}
\label{thm:H3}
Let $T$ be a tree and $i\ge 3$.
Then the equation $W(L^i(T))=W(T)$ has a solution if and only if $i=3$ and
$G$ is of type $H_{a,b,c}$ as stated in Theorem~{\ref{thm:H}}.
\end{theorem}

We conclude this section with the following problem.

\begin{problem}
Find all graphs (with cycles) $G$ and powers $i$ for which
\begin{equation}
\label{eq:iter}
  W(L^i(G))=W(G).
\end{equation}
\end{problem}

For $i=1$ the above problem is very rich with many different solutions,
so probably it will not be possible to find all of them.
But still, stating it as a problem could serve as a motivation for searching
of various graph classes that satisfy the equation.
However, we want to emphasize the case $i\ge 2$.
In this case the problem is still rich with many solutions, particularly
among the trees, but abandoning the class of trees can reduce the solutions
significantly.
At the moment, cycles are the only known cyclic graphs $G$ for which
$W(L^i(G))=W(G)$ holds for some $i\ge 3$ and we believe that there are no
other cyclic graphs satisfying (\ref{eq:iter}). This was conjectured independently in
\cite{DM3} and  \cite{KS_chapter}.

\begin{conjecture}
Let $i\ge 3$.
There is no graph $G$, distinct from a cycle and a tree, such that
$$
  W(L^i(G))=W(G).
$$
\end{conjecture}

%
%
%
%
%

\section{Excursion into digraphs}

In~\cite{KST, KST2}, we have considered  the Wiener index of not
necessarily strongly connected digraphs. In order to do so, if in a
digraph there is no directed path from a vertex $u$ to a vertex $v$,
we follow the convention that \begin{equation} \label{zero}
d(u,v)=0,
\end{equation}
which was independently introduced in several studies of directed networks.

A counterpart of the Wiener theorem for {\em directed
trees}, i.e. digraphs whose underlying graphs are trees can be stated in this way.

\begin{theorem} \label{win} Let $T$ be a directed tree with the arc set $A(T)$.
Then $$W(T)=\displaystyle \sum_{ab \in A(T)}t(a)s(b),$$
where $t(a)$ denotes the number of vertices reachable from $a$, and $s(b)$ denotes
the number of vertices that reach $b$.
\end{theorem}

 Here we give a counterpart of a relation between the Wiener index and betweenness centrality $B(x)$ for oriented graphs.

\begin{theorem} For any digraph $D$ of order $n$
$$W(D) = \sum_{x \in
V(D)}B(x) + p(D),$$
where $p(D)$ denotes the number of ordered pairs $(u,v)$ such that there
exists a directed path from $u$ to $v$ in $D$.
\end{theorem}

\noindent The above result implies that for strongly connected digraph $D$ on $n$ vertices, we have the relation
    $$W(D) = \sum_{x \in V(D)}B(x) + 2\binom{n}{2}.$$

Let $W_{\max}(G)$ and $W_{\min}(G)$ be the maximum possible and the minimum
possible, respectively, Wiener index among all digraphs
obtained by orienting the edges of a graph $G$.

\begin{problem}
For a given graph $G$ find  $W_{\max}(G)$ and $W_{\min}(G)$.
\end{problem}

The above problem has been considered for strongly connected
orientations. Plesn{\'\i}k \cite{P} proved that finding a strongly
connected orientation of a given graph $G$ that minimizes the Wiener
index is NP-hard. Regarding the problem of finding $W_{\max}(G)$,
Plesn{\'\i}k and Moon \cite{moon,P} resolved it for complete graphs,
under the assumption that the orientation is strongly connected.

We showed~\cite{KST} that the above mentioned results of Plesn{\'\i}k and Moon hold also
for non-strongly connected orientations assuming the condition (\ref{zero}).
One may expect that {\em for a 2-connected graph $G$, $W_{\max}(G)$ is attained
for some strongly connected orientation}. However, this is not the case as
we proved by $\Theta$-graphs $\Theta_{a,b,1}$ for $a$ and $b$ fulfilling certain conditions. By $\Theta_{a,b,c}$ we denote a graph obtained when
two distinct vertices are connected by three
internally-vertex-disjoint paths of lengths $a+1$, $b+1$ and $c+1$,
respectively. We assume $a\geq b\geq c$ and $b\geq 1$.
The orientation of $\Theta_{a,b,c}$ which achieves the maximum Wiener index is not
strongly connected if $c\ge 1$.
However, we believe that the following holds.

\begin{conjecture}
Let $a\geq b \geq c$.
Then $W_{\max}(\Theta_{a,b,c})$ is attained by an orientation of
$\Theta_{a,b,c}$ in which the union of the paths of lengths $a+1$
and $b+1$ forms a directed cycle.
\end{conjecture}

Analogous results as for $\Theta$-graphs, stating that the orientation of a graph which achieves
the maximum Wiener index is not strongly connected, can probably be proved
also for other graphs which are not very dense and which admit an orientation
with one huge directed cycle without ``shortcuts", that is without directed
paths shortening the cycle. On the other hand, we were not able to find examples without long induced cycles
that makes us wonder if the following holds.

\begin{conjecture}
Let $G$ be a 2-connected chordal graph. Then $W_{\max}(G)$ is attained by an orientation
which is strongly connected.
\end{conjecture}

Finally, we wonder how hard it is to find $W_{\max}$ and $W_{\min}$.

\begin{problem}
For a given graph G, what is the complexity of finding
$W_{\max}(G)$  (resp. $W_{\min}(G)$)? Are these problems NP-hard?
\end{problem}

Consider also the following problem for the minimum value.

\begin{conjecture} For every graph $G$, the value $W_{\min}(G)$ is achieved for some acyclic orientation $G$.
\end{conjecture}

This is certainly true for bipartite graphs. Namely, by orienting
all edges of such a graph $G$ so that the corresponding arcs go from
one bipartition to the other, we obtain a digraph $D$ with
$W(D)=|E(G)|$. As obviously $W_{\min}(G)\ge |E(G)|$, this case is
established.

%
%
%
%
%

\section{Wiener index for disconnected graphs}

Since the formula (\ref{WienerIndex}) cannot be applied to non-connected
graphs, for these graphs we set
\begin{equation}
\label{eq:new2}
  W(G)=\sum_{\substack{\{x,y\}\subseteq V(G)\\ x-y {\rm \ path\ exists\ in\ }G}} d(x,y).
\end{equation}
In other words, we ignore pairs of vertices $x$ and $y$ for which the distance $d(u,v)$
can be considered as ``infinite" analogously as we ignored such pairs of
vertices in the case of digraphs. For example, in \cite{dobr-9},  the Wiener index has been used in quantitative studies of disconnected hexagonal networks.

Let $G$ be a disconnected graph with components $G_1$, $G_2$, \dots, $G_p$.
By (\ref{eq:new2}) we get
$$
W(G) = W(G_1) + W(G_2) + \dots + W(G_p) \,.
$$

It is interesting to study the problems from the previous
sections using the modified definition of Wiener index (\ref{eq:new2}).
Particularly, we find interesting the analogues of
Problems~\ref{prob:inv_gen} and~(\ref{prob:ite_target}).

\begin{problem}
For given $n$, find all values $w$ which are Wiener indices of
not necessarily connected graphs (forests) on $n$ vertices.
\end{problem}

Let $i\ge 3$.
From the proof of Theorem~{\ref{thm:H3}} one can see that most trees $T$
satisfy $W(L^i(T))>W(T)$, while paths on $n\ge 2$ vertices satisfy
$W(L^i(P_n))<W(P_n)$.
Hence, the following problem is interesting.

\begin{problem}
For $i\ge 3$, find all forests $F$ for which $W(L^i(F))=W(F)$.
\end{problem}

%
%
%
%
%

\section{Trees with given degree conditions}

Lin \cite{Lin2} characterized the trees which maximize and
minimize the Wiener index among all trees of given order that have only vertices
of odd degrees. An ordering of trees by their smallest Wiener indices for trees of given order that have only vertices
of odd degrees was obtained by Furtula, Gutman and Lin \cite{furt}. In \cite{furt2} Furtula further determined the trees
with the second up to seventeenth greatest Wiener indices. Lin \cite{Lin2} suggested analogous problems for general graphs.

\begin{problem}
Characterize the graphs with maximal Wiener index in the set of
graphs on $2n$ vertices whose vertices are all of odd degree, and in
the set of  graphs on $n$ vertices whose vertices are all of even
degree, respectively.
\end{problem}

In \cite{Lin3} Lin characterized the trees which minimize (maximize,
respectively) the Wiener index among all trees with given number of
vertices of even degree. He proposed the following problems for the
class of graphs $E_{n,r}$ of order $n$ with exactly $r$ vertices of
even degree, where $r\geq 1$ and $n\equiv r \pmod 2$.

\begin{problem}
Order the trees in $E_{n,r}$  with the smallest or greatest Wiener index.
\end{problem}

\begin{problem}
Characterize graphs with maximal and minimal Wiener index in $E_{n,r}$, respectively.
\end{problem}

The same author in \cite{Lin4} characterized trees which maximize the Wiener index among all trees of order $n$
with exactly $k$ vertices of maximum degree. For better understanding how the maximum degree vertices influence the Wiener
index he proposes to consider analogous problem for the minimum.

\begin{problem}
Characterize the tree(s) with the minimal Wiener index among all trees of order $n$
with exactly $k$ vertices of maximum degree.
\end{problem}

Wang \cite{wang} and Zhang et al. \cite{zhang1} independently
determined the tree that minimizes the Wiener index among trees of
given degree sequence. But the following problem from \cite{jin,
sills, zhang2} is still open, although it is known for longer time
that extremal graphs are caterpillars~\cite{Shi}.

\begin{problem}
Which tree maximizes the Wiener index among trees of given degree sequence?
\end{problem}

\section{Few more problems}
Here we collect some more problems on Wiener index.
\paragraph{Eulerian graphs.}
Denote by ${\cal E}_{n}$ the set of all Eulerian graphs of order $n$.
Gutman et al.~\cite{Gut-eul} characterized elements of  ${\cal E}_{n}$
having the first few smallest Wiener indices.
They proved that for graphs in ${\cal E}_{n}$, $C_n$ attains
the maximal value.
In addition, they posed a conjecture on the second-maximal Wiener index
in ${\cal E}_{n}$.

\begin{conjecture}
The second-maximal Wiener index between all  Eulerian graphs of large
enough order $n$ is attained by $C_{n,3}$ (i.e. the graph obtained
from disjoint cycles $C_{n-2}$ and $C_3$ by identifying one vertex in each of them).
\end{conjecture}

They have also analogous conjecture for small values of $n$,
see \cite{Gut-eul} for more details.


\paragraph{Fullerene graphs.}

In \cite{hua} the Wiener indices of the $(6,0)$-nanotubes (tubical
fullerenes) is computed. Note that such a graph has $12k$ vertices,
for some $k\ge 2$, and the corresponding value of the Wiener index
is $48k^3 + 828k-1632$. These fullerenes have long diameter and
consequently big Wiener index. Nevertheless they believe that the
following may hold.

\begin{conjecture}
Wiener index of fullerene graphs on $n$ vertices is of asymptotic order $\theta(n^3)$.
\end{conjecture}


\paragraph{Wiener index versus Szeged index.}
Klav\v zar, Rajapakse and Gutman~\cite{KRG} showed that $\Sz(G) \ge W(G)$, and even more, by a result of
Dobrynin and Gutman~\cite{DGu}, equality $\Sz(G) = W(G)$ holds if and only if each block of $G$ is complete.
In \cite{NKA1} a classification of graphs with $\eta(G) =\Sz(G) - W(G)\le 3$ is presented.
In \cite{NKA2} the authors classify connected graphs which satisfy $\eta(G) = 4$ or $5$.
Moreover, they state the following conjecture.

\begin{conjecture}
Let $G$ be a graph of order $n$ with blocks $B_1,\ldots,B_k$ such that none is complete.
Let $B_i$ be of order $n_i$. Then
     $$ \Sz(G) - W(G) \ge \sum_{i=1}^k (2n_i-6). $$
\end{conjecture}
The difference $\eta$ was also studied by Klav\v zar and Nadjafi-Arani~\cite{KN}.


\paragraph{Wiener index of graphs with given matching number.}
Zhou and Trinajsti\'{c} \cite{ZhT}  determined the minimum Wiener index of connected graphs with $n\geq 5$ vertices and matching
number $i\geq 2$, and characterized the extremal graphs. Du and Zhou \cite{du} determined the minimum Wiener indices
of trees and unicyclic graphs, respectively, with given number of vertices and matching number. Also, they characterized extremal graphs. For this class of trees Tan \cite{Tan} et al. determined ordering of trees with the smallest Wiener indices.

Regarding the maximum Wiener index, Dankelmann \cite{dan} determined it for connected graphs with $n\geq 5$
vertices and matching number $i\geq 2$, and he characterized the unique extremal graph,
which turned out to be a tree.
Thus, the maximum Wiener index among trees with given number of vertices and matching number
is known, as well as the corresponding unique extremal graph.
Finding the maximum Wiener index among unicyclic graphs remains an open problem \cite{du}.

\begin{problem}
Find the maximum Wiener index among unicyclic graphs with $n$ vertices and matching
number $i$ for $3\leq i\leq \left\lfloor \frac{n}{2} \right\rfloor -1$.
\end{problem}

\paragraph{Graph connectivity.}
Graphs with higher connectivity have more edges, and
henceforth smaller Wiener index.
Gutman and Zhang~\cite{Gut-Zhang} showed that in the class of
$k$-connected graphs on $n$ vertices, the minimum value of Wiener index
is attained by $K_k+(K_1\cup K_{n-k-1})$, i.e. the graph obtained when we connect
all vertices of $K_k$ with all vertices of disjoint union of $K_1$ and $K_{n-k-1}$.
This graph is extremal also in the class of $k$-edge-connected graphs on $n$ vertices.
They pose the following problem.

\begin{problem}
Find the maximum Wiener index among $k$-connected graphs on $n$ vertices.
\end{problem}

Note that  $P_n$ is the extremal graph in the class of $1$-connected graphs,
and $C_n$ is extremal in the class of $2$-connected graphs.
Of course, similar problem can be posed for $k$-edge-connected graphs.
The authors of \cite{Gut-Zhang} ask whether the extremal graphs for
the maximum Wiener index in the classes of $k$-connected and
$k$-edge-connected graphs coincide, as is the case for
the minimum Wiener index.

\paragraph{Trees and unicyclic graphs with given bipartition.}
Du~\cite{dudu} considered Wiener index of trees and unicyclic graphs
on $n$ vertices with prescribed sizes of bipartitions $p$ and $q$,
where $n=p+q$ and $p\ge q$.
He showed that in the case of trees, the extremal graph for
the minimum Wiener index is obtained by connecting the centers
of disjoint stars $K_{1,p-1}$ and $K_{1,q-1}$, and the extremal graph
for the maximum Wiener index is obtained by connecting the end-vertices
of a path $P_{2q-1}$ with $\lceil(p-q+1)/2 \rceil$ and
$\lfloor(p-q+1)/2 \rfloor$ new vertices, respectively.

Regarding the unicyclic graphs, Du showed that the minimum Wiener index
is attained by the graph, which is obtained by connecting $p-2$ vertices
to one vertex of a $4$-cycle, and connecting $q-2$ vertices to its neighbour
on the $4$-cycle.
Moreover, if $p=q=3$, then $C_6$ is also an extremal graph.
What remains open, is the maximum value.

\begin{problem}
Find the maximum Wiener index among unicyclic graphs on $n$ vertices with
bipartition sizes $p$ and $q$, where $n=p+q$.
\end{problem}

\bigskip
\noindent
{\bf Acknowledgment.} We are thankful to A.~A.~Dobrynin for his valuable
comments and suggestions that improved the paper.
The first author acknowledges partial support by
Slovak research grants VEGA 1/0781/11, VEGA 1/0065/13 and APVV 0136-12.
All authors are partially supported by Slovenian research agency ARRS,
program no.\ P1--00383, project no.\ L1--4292,
and Creative Core--FISNM--3330-13-500033.

\bigskip


\bibliographystyle{plain}

\end{document}